\let\chooseClass1   %Matematychni Studii journal
\let\chooseClass2   %extarticle or article for ourselves
\let\chooseClass3   %article, 12pt, for arXiv

\ifx\chooseClass1
\documentclass[12pt, twoside, draft]{article}
\usepackage[cp1251]{inputenc}
\newcommand{\CopyName}{V.\ V.\ Lyubashenko} % Author's name
\newcommand{\NAME}{V.\ V.\ LYUBASHENKO} %
\newcommand{\Volume}{?} % Volume (to fill in by the Editors)
\newcommand{\Number}{?} % Issue (to fill in by the Editors)
\newcommand{\rightheadtext}{Curved bar and cobar constructions} %please give short version of the title, not exceeding  40 symbols
\usepackage[english,russian,ukrainian]{babel}
\usepackage{ms_we2}
\renewcommand{\refname}{\refnam}
\English
\vs
\newcommand{\tit}{Bar and cobar constructions for curved algebras and coalgebras} % A title of the paper.
\date{}
\fi

\ifx\chooseClass2
\documentclass[14pt]{extarticle}
\fi

\ifx\chooseClass3
\documentclass[12pt]{article}
\fi

\ifx\chooseClass1
	\else
\makeatletter
\def\@seccntformat#1{\csname the#1\endcsname.\quad}
\renewcommand\section{\@startsection {section}{1}{\z@}%
                                   {-3.5ex \@plus -1ex \@minus -.2ex}%
                                   {2.3ex \@plus.2ex}%
                                   {\normalfont\large\bfseries}}
\renewcommand\subsection{\@startsection{subsection}{2}{\z@}%
                        {3.25ex plus 1ex minus .2ex}{-.5em}%
                        {\normalfont\normalsize\bfseries}}
\renewcommand\subsubsection{\@startsection{subsubsection}{3}{\z@}%
                        {3.25ex plus 1ex minus .2ex}{-.5em}%
                        {\normalfont\normalsize\it}}
\@addtoreset{equation}{section}
\makeatother

\usepackage{amsthm}
\newtheoremstyle{boldhead}%     name
{\topsep}%                      abovespace
{\topsep}%                      belowspace
{\slshape}%                     bodyfont
{}%                             indentation=noindent
{\bfseries}%                    headfont
{.}%                            headpunctuation
{ }%                            headspace=interword space
{\thmname{#1}\thmnumber{ #2}\thmnote{ (#3)}}%   custom head specification

\newtheoremstyle{boldremark}%   name
{\topsep}%                      abovespace
{\topsep}%                      belowspace
{\upshape}%                     bodyfont
{}%                             indentation=noindent
{\bfseries}%                    headfont
{.}%                            headpunctuation
{ }%                            headspace=interword space
{\thmname{#1}\thmnumber{ #2}\thmnote{ (#3)}}%   custom head specification

\swapnumbers

\theoremstyle{boldhead}
\newtheorem{theorem}[subsection]{Theorem}
\newtheorem{corollary}[subsection]{Corollary}

\newtheorem{proposition}[subsection]{Proposition}

\theoremstyle{boldremark}
\newtheorem{definition}[subsection]{Definition}
\newtheorem{example}[subsection]{Example}

\fi%\chooseClass2,3

\usepackage{    %abbrevs,
				amsmath,
                amsfonts,
                amssymb,
                amsxtra,
}

\ifx\chooseClass1
	\else
\usepackage[matha,mathx]{mathabx}
\fi

\numberwithin{equation}{section}

\usepackage{ifpdf}
\ifpdf
	\usepackage[pdftex]{hyperref}
\else
    \usepackage[hypertex]{hyperref}
\fi

\usepackage{diagrams}
\diagramstyle[height=2em,balance,righteqno,PostScript=dvips,nohug]

\newarrow{DashTo}{}{dash}{}{dash}{->}
\newarrow{MapsTo}{mapsto}---{->}
\newarrow{Mono}{boldhook}---{->}
\newarrow{TTo}----{->}

\newcommand\NN{{\mathbb N}}
\newcommand\RR{{\mathbb R}}

\newcommand{\sfj}{{\mathsf j}}
\newcommand{\sfr}{{\mathsf r}}
\newcommand{\sfv}{{\mathsf v}}
\newcommand{\sfw}{{\mathsf w}}

\newcommand{\beps}{{\boldsymbol\eps}}
\newcommand{\bfeta}{{\boldsymbol\eta}}

\newcommand{\bull}{{\scriptscriptstyle\bullet}}

\newcommand{\bv}{{\mathbf v}}
\newcommand{\bw}{{\mathbf w}}
\newcommand{\tdt}{\otimes\dots\otimes}

\newcommand{\n}[1]{\nobreakdash-\hspace{0pt}}
\newcommand{\ainf}[1]{$A_\infty$\nobreakdash-\hspace{0pt}}
\newcommand{\ainfm}[1]{$\mathrm{A}_\infty$\nobreakdash-\hspace{0pt}}

\let\eps\varepsilon
\let\ge\geqslant
\let\kk\Bbbk

\let\rto\xrightarrow

\let\tens\otimes

\let\und\underline

\newcommand\sff{{\mathsf f}}
\newcommand\sfg{{\mathsf g}}
\newcommand\sfh{{\mathsf h}}
\newcommand\sfq{{\mathsf q}}

\newcommand{\alg}{\textup{-alg}}

\DeclareMathOperator{\Bbar}{Bar}
\DeclareMathOperator{\sfBar}{{\sf Bar}}
\DeclareMathOperator{\Cobar}{Cobar}
\DeclareMathOperator{\sfCobar}{{\sf Cobar}}
\DeclareMathOperator\dg{\mathbf{dg}}
\DeclareMathOperator\End{End}
\DeclareMathOperator\ev{ev}
\DeclareMathOperator\gr{\mathbf{gr}}
\DeclareMathOperator\id{id}

\DeclareMathOperator\inj{in}
\DeclareMathOperator\Ker{Ker}
\DeclareMathOperator{\nuCoalg}{nuCoalg}
\DeclareMathOperator\Ob{Ob}
\DeclareMathOperator\opr{\overline{\textup{pr}}}
\DeclareMathOperator\pr{pr}
\DeclareMathOperator{\Tw}{Tw}
\DeclareMathOperator{\UCCAlg}{UCCAlg}
\DeclareMathOperator{\uccAlg}{uccAlg}
\DeclareMathOperator{\ucdgAlg}{uc\mathbf{dg}Alg}
\DeclareMathOperator{\acCoalg}{acCoalg}
\DeclareMathOperator{\CACoalg}{CACoalg}
\DeclareMathOperator{\caCoalg}{caCoalg}

\newcommand{\defref}[1]{Definition~\ref{#1}}

\newcommand{\thmref}[1]{Theorem~\ref{#1}}

\begin{document}
\ifx\chooseClass1
\hbox to \textwidth{\footnotesize\textsc  Математичні Студії. Т.\Volume, \No \Number
\hfill
Matematychni Studii. V.\Volume, No.\Number}
\vspace{0.3in}
\textup{\scriptsize{УДК 512.582}} \vs % Індекс УДК.
\markboth{{\NAME}}{{\rightheadtext}}
\begin{center} \textsc {\CopyName} \end{center}
\begin{center} \renewcommand{\baselinestretch}{1.3}\bf {\tit} \end{center}

\vspace{20pt plus 0.5pt} {\abstract{ \noindent V.\ V.\ Lyubashenko.\ % Author's name (in English).
\textit{Bar and cobar constructions for curved algebras and coalgebras}\matref \vspace{3pt} \English % A title of the paper (in English).

We provide bar and cobar constructions as functors between some categories of curved algebras and curved augmented coalgebras over a graded commutative ring.
These functors are adjoint to each other.

\bigskip\Russian
\noindent В.\ В.\ Любашенко.\ % Author's name (in Russian).
\textit{Бар и кобар конструкции для кривых алгебр и коалгебр}\matrefrus \vspace{3pt} % A title of the paper in (in Russian).

Мы рассматриваем бар и кобар конструкции как функторы между некоторыми категориями кривых алгебр и кривых увеличенных коалгебр над градуированным коммутативным кольцом.
Эти функторы сопряжены друг с другом.

}} \vsk
 %  2000 AMS Mathematics Subject Classification.

 %Keywords of the paper (in English)

\renewcommand{\refname}{\refnam}
\renewcommand{\proofname}{\ifthenelse{\value{lang}=0}{Proof}{\ifthenelse{\value{lang}=2}{Доказательство}{Доведення}}}

\vskip10pt
	\else
\bibliographystyle{amsalpha}
\title{Bar and cobar constructions for curved algebras and coalgebras}
\author{Volodymyr Lyubashenko}
\maketitle

\begin{abstract}
We provide bar and cobar constructions as functors between some categories of curved algebras and curved augmented coalgebras over a graded commutative ring.
These functors are adjoint to each other.
\end{abstract}
\fi

\allowdisplaybreaks[1]

In this article we recall some notions and reproduce some results from Positselski \cite{0905.2621,1202.2697} in modified form.
Our exposition differs in two aspects: firstly, we work over a graded commutative ring $\kk$ instead of a field or a topological local ring, secondly, we modify the definitions of categories of curved algebras and curved coalgebras.

The advantage of using graded commutative rings over usual commutative rings is that it allows to place (co)derivations of certain degree on equal footing with (co)algebra homomorphisms.
Take note of the last condition in the following

\begin{definition}\label{def-graded-commutative-ring}
A \emph{graded strongly commutative ring} is a graded ring $\kk$ such that \(ba=(-1)^{|a|\cdot|b|}ab\) for all homogeneous elements $a$, $b$ and $c^2=0$ for all elements $c$ of odd degree.
\end{definition}

The first condition implies only that $2c^2=0$ for elements $c$ of odd degree.

We give explicit formulae and detailed proofs.
Motivations come from \ainf-algebras and \ainf-coalgebras.

For any graded $\kk$-module $M$ and an integer $a$ denote by $M[a]$ the same module with the grading shifted by $a$: \(M[a]^k=M^{a+k}\).
Denote by \(\sigma^a:M\to M[a]\), \(M^k\ni x\mapsto x\in M[a]^{k-a}\) the ``identity map'' of degree \(\deg\sigma^a=-a\).
Write elements of $M[a]$ as \(m\sigma^a\).
Typically, a map is written on the right of its argument.
The composition of $X\rto f Y\rto g Z$ is denoted by $f\cdot g:X\to Z$ or simply by $fg$.
When \(f:V\to X\) is a homogeneous map of certain degree, the map \(f[a]:V[a]\to X[a]\) is defined as \(f[a]=(-1)^{a\deg f}\sigma^{-a}f\sigma^a=(-1)^{af}\sigma^{-a}f\sigma^a\).
The tensor product of homogeneous maps $f$, $g$ between graded $\kk$\n-modules is defined on elements $x$, $y$ of certain degree as
\[ (x\tens y).(f\tens g) = (-1)^{\deg y\cdot\deg f}x.f\tens y.g.
\]
Thus, the Koszul sign rule holds and we work in the closed symmetric monoidal category $\gr$ of graded $\kk$\n-modules with the symmetry \(x\tens y\mapsto(-1)^{\deg x\cdot\deg y}y\tens x\).

\section{Curved (co)algebras}
We define curved algebras and curved coalgebras as well as their morphisms suitable for our purposes.

\subsection{Curved algebras.}
We begin with curved algebras of various kinds.

\begin{definition}\label{def-strict-unit-complemented-curved-A8-algebra}
A \emph{strict-unit-complemented curved \ainf-algebra} \((A,(b_n)_{n\ge0},\bfeta,\bv)\) consists of a graded $\kk$\n-module $A$, degree 1 maps \(b_n:A[1]^{\tens n}\to A[1]\) (operations) for $n\ge0$, a degree $-1$ map \(\bfeta:\kk\to A[1]\) (strict unit) and a degree 1 map \(\bv:A[1]\to\kk\) (splitting of the unit) such that
\begin{gather}
\sum_{r+k+t=n} (1^{\tens r}\tens b_k\tens1^{\tens t})b_{r+1+t} =0: A[1]^{\tens n} \to A[1], \qquad \forall\, n\ge0,
\label{eq-b-b-0}
\\
(1\tens\bfeta)b_2 =1_{A[1]}, \quad (\bfeta\tens1)b_2 =-1_{A[1]}, \quad (1^{\tens a}\tens\bfeta\tens1^{\tens c})b_{a+1+c} =0 \text{ \ if \ } a+c\ne1, \notag
%\label{eq-(1bone)b2-1-(bone1)b2-1}
\\
\bfeta \cdot \bv =1_\kk. \notag
\end{gather}
\end{definition}

For any graded $\kk$\n-module $X$ the tensor $\kk$\n-module \(XT^\ge=\oplus_{n\ge0}X^{\tens n}\) is equipped with the cut coproduct
\[ (x_1\cdots x_n)\Delta =\sum_{k=0}^n x_1\cdots x_k \tens x_{k+1}\cdots x_n.
\]
The collection \(\check b=(b_n)_{n\ge0}:A[1]T^\ge\to A[1]\) amounts to a degree 1 coderivation \(b:A[1]T^\ge\to A[1]T^\ge\) of the counital coassociative coalgebra \(A[1]T^\ge\),
\[ b| =\sum_{r+k+t=n} 1^{\tens r}\tens b_k\tens1^{\tens t}:
A[1]^{\tens n}\to A[1]T^\ge.
\]
Equation~\eqref{eq-b-b-0} is equivalent to $b^2=0$.

Getting rid of the shift $[1]$ we rewrite the above operations as in \cite[(0.7)]{Lyu-A8-several-entries}
\begin{gather*}
m_n =(-1)^n\sigma^{\tens n}\cdot b_n\cdot\sigma^{-1}: A^{\tens n} \to A, \qquad \deg m_n=2-n, \qquad n\ge0,
\\
\eta =\bigl( \kk \rTTo^{\bfeta} A[1] \rTTo^{\sigma^{-1}} A \bigr), \qquad \deg \eta =0,
\\
\sfv =\bigl( A \rTTo^\sigma A[1] \rTTo^\bv \kk \bigr), \qquad \deg \sfv =0.
\end{gather*}
In these terms \defref{def-strict-unit-complemented-curved-A8-algebra} becomes

\begin{definition}
A \emph{strict-unit-complemented curved \ainfm-algebra} \((A,(m_n)_{n\ge0},\eta,\sfv)\) consists of a graded $\kk$\n-module $A$, maps \(m_n:A^{\tens n}\to A\) of degree $2-n$ (operations) for $n\ge0$, a degree 0 map \(\eta:\kk\to A\) (strict unit) and a degree 0 map \(\sfv:A\to\kk\) (splitting of the unit) such that
\begin{gather}
\sum_{j+p+q=n} (-1)^{jp+q}(1^{\tens j}\tens m_p\tens1^{\tens q})\cdot m_{j+1+q} =0: A^{\tens n} \to A, \qquad \forall\, n\ge0,
\label{eq-1m1-m-0}
\\
(1\tens\eta)m_2 =1_A, \quad (\eta\tens1)m_2 =1_A, \quad
(1^{\tens a}\tens\eta\tens1^{\tens c})m_{a+1+c} =0 \text{ \ if \ } a+c\ne1, \notag
\\
\eta \cdot \sfv =1_\kk. \notag
\end{gather}
\end{definition}

Restricting the above notion we get

\begin{definition}
A \emph{unit-complemented curved algebra} \((A,m_2,m_1,m_0,\eta,\sfv)\) is a strict-unit-complemented curved \ainfm-algebra $A$ with the strict unit $\eta$ and with $m_n=0$ for $n>2$.
\end{definition}

For such algebra $A$ equations~\eqref{eq-1m1-m-0} reduce to the system
\begin{gather*}
(1\tens m_2)m_2 =(m_2\tens1)m_2, \qquad m_2m_1 =(1\tens m_1 +m_1\tens1)m_2,
\\
m_1^2 =(m_0\tens1 -1\tens m_0)m_2, \qquad m_0m_1 =0,
\\
(1\tens\eta)m_2 =1, \qquad (\eta\tens1)m_2 =1, \qquad \eta m_1 =0, \qquad \eta \sfv =1_\kk,
\end{gather*}
which tells that $A$ is a unital associative graded algebra \((A,m_2,\eta)\) with a degree~1 derivation $m_1$, whose square is an inner derivation -- commutator with an element $m_0$ (curvature) of degree~2 and $m_0m_1=0$.
A direct complement \(\bar{A}=\Ker\sfv\) to the $\kk$\n-submodule \(\eta:\kk\hookrightarrow A\) is chosen.

The following example of a unit-complemented curved algebra was considered by Positselski in \cite[Section~0.6]{0905.2621}, see also Polishchuk \cite{Polishchuk:curved-dg}.

\begin{example}
Let $M$ be a smooth manifold, let $E\to M$ be a smooth vector bundle, $\kk=\RR$, denote
\[ \Omega^k(E) =\Gamma(E\tens\wedge^kT^*M), \qquad k\in\NN.
\]
Let \(\nabla:\Omega^0(E)\to\Omega^1(E)\) be a connection on $E$ which is viewed as a covariant exterior derivative \(\nabla:\Omega^k(E)\to\Omega^{k+1}(E)\) such that
\[ \forall\,\tau\in\Omega^\bull(E) \quad \forall\,\omega\in\Omega^\bull(M) \qquad (\tau\omega)\nabla =(-1)^\omega(\tau\nabla)\cdot\omega +\tau\cdot(\omega d).
\]

The category of vector bundles on $M$ is Cartesian closed.
The evaluation map \(\ev:E\times\End E\to E\) leads to the action \(\Omega^k(E)\tens\Omega^n(\End E)\to\Omega^{k+n}(E)\).
Moreover, elements \(h\in A^n=\Omega^n(\End E)\) can be identified with $\Omega^\bull(M)$-linear maps \(h:\Omega^k(E)\to\Omega^{k+n}(E)\), thus, \((\tau\omega)h=(-1)^{n|\omega|}(\tau h)\omega\).
For instance, the curvature 2-form \(-m_0=\nabla^2\) is a $\Omega^\bull(M)$-linear map, hence an element of \(\Omega^2(\End E)\).

The graded algebra \(A^\bull=\Omega^\bull(\End E)\) equipped with the derivation \((h)d_A=h\cdot\nabla-(-1)^h\nabla\cdot h\) (which is a covariant exterior derivative on the vector bundle $\End E$) and with the curvature element \(m_0\in A^2\) is a curved algebra since
\[ (h)d_A^2 =m_0h -hm_0, \qquad (m_0)d_A =0.
\]
The latter equation is the Bianchi identity.
\end{example}

A morphism between curved \ainf-algebras $A$ and $B$ should be given by a family of components \(f_n:A[1]^{\tens n}\to B[1]\), $n\ge0$.
The obtained matrix entries
\[ f_n^k =\sum_{i_1+\dots+i_k=n} f_{i_1}\tens f_{i_2}\tdt f_{i_k}: A[1]^{\tens n}\to B[1]^{\tens k}
\]
define a map \(f:A[1]T^\ge\to B[1]\hat T^\ge\), which in general does not factor through \(B[1]T^\ge\).
The equation $fb=bf$, which we write as
\[ \sum_{i_1+\dots+i_k=n} (f_{i_1}\tens f_{i_2}\tdt f_{i_k})b^B_k =\sum_{r+k+t=n} (1^{\tens r}\tens b^A_k\tens1^{\tens t})f_{r+1+t},
\]
also makes sense under some additional assumptions (like extra filtration \cite{FukayaOhOhtaOno:Anomaly} or topological structure of $\kk$ \cite{1202.2697}).
We shall consider only curved algebras $B$, which insures that the sum in the left hand side is finite.
Moreover, we assume that components $f_n$ vanish for $n>1$ and $f_0$ is of the form
\begin{equation}
f_0 =\bigl( \kk \rTTo^{\und f} \kk \rTTo^{\bfeta} B[1] \bigr),
\label{eq-f0-undf-1}
\end{equation}
where \(\deg\und f=1\).
The latter assumption was made in order to deal with augmented coalgebras in bar and cobar constructions.
Which does not exclude that similar results could be obtained under weaker assumptions.

\begin{definition}
A \emph{morphism of unit-complemented curved algebras} \(f:A\to B\) is a pair \((f_1,\und f)\) consisting of $\kk$\n-linear maps \(f_1:A[1]\to B[1]\) of degree~0 and \(\und f:\kk\to\kk\) of degree~1 such that
\begin{equation}
(f_1\tens f_1)b^B_2 =b^A_2f_1, \qquad f_1b^B_1 =b^A_1f_1, \qquad
b^B_0 =b^A_0f_1, \qquad \bfeta_Af_1 =\bfeta_B.
\label{eq-fffb-bf}
\end{equation}
The composition $h:A\to C$ of morphisms \(f:A\to B\) and \(g:B\to C\) is \(h_1=f_1g_1\), \(\und h=\und g+\und f\).
\end{definition}

Under assumption \eqref{eq-f0-undf-1} the expected conditions
\begin{gather*}
f_1b^B_1 +(f_0\tens f_1)b^B_2 +(f_1\tens f_0)b^B_2 =b^A_1f_1,
\\
b^B_0 +f_0b^B_1 +(f_0\tens f_0)b^B_2 =b^A_0f_1, \qquad h_0=g_0+f_0g_1
\end{gather*}
reduce to the given ones.
In fact, \(\und f\in\kk^1\) implies \(\und f^2=0\) due to graded commutativity of $\kk$, see \defref{def-graded-commutative-ring}.

The last equation of \eqref{eq-fffb-bf} tells that $f_1$ preserves the unit.
These equations can be rewritten for conventional $\kk$\n-linear maps
\begin{alignat*}2
\sff_1 &=\bigl( A \rTTo^\sigma A[1] \rTTo^{f_1} B[1] \rTTo^{\sigma^{-1}} B \bigr), &\qquad \deg \sff_1 &=0,
\\
\sff_0 &=\bigl( \kk \rTTo^{f_0} B[1] \rTTo^{\sigma^{-1}} B \bigr) =\bigl( \kk \rTTo^{\und f} \kk \rTTo^\eta B \bigr), &\qquad \deg \sff_0 &=1,
\end{alignat*}
as follows.

\begin{definition}
A \emph{morphism of unit-complemented curved algebras} \(f:A\to B\) is a pair \((\sff_1,\und f)\) consisting of $\kk$\n-linear maps \(\sff_1:A\to B\) of degree~0 and \(\und f:\kk\to\kk\) of degree~1 such that
\begin{equation*}
(\sff_1\tens \sff_1)m^B_2 =m^A_2\sff_1, \qquad \sff_1m^B_1 =m^A_1\sff_1, \qquad m^B_0 =m^A_0\sff_1, \qquad \eta^A\sff_1 =\eta^B.
\end{equation*}
The composition $\sfh:A\to C$ of morphisms \(\sff:A\to B\) and \(\sfg:B\to C\) is \(\sfh_1=\sff_1\sfg_1\), \(\und h=\und g+\und f\).
The unit morphism is \((\id,0)\).
The category of unit-complemented curved algebras is denoted $\UCCAlg$.
\end{definition}

In particular, \(\sff_1:A\to B\) is a morphism of unital associative graded algebras.

\subsection{Curved coalgebras.}
Now we define curved coalgebras of various kinds.

\begin{definition}\label{def-strict-counit-complemented-curved-A8-coalgebra}
A \emph{strict-counit-complemented curved \ainf-coalgebra} \((C,(\xi_n)_{n\ge0},\beps,\bw)\) consists of a graded $\kk$\n-module $C$, degree 1 maps \(\xi_n:C[-1]\to C[-1]^{\tens n}\) (cooperations) for $n\ge0$, a degree $-1$ map \(\beps:C[-1]\to\kk\) (strict counit) and a degree 1 map \(\bw:\kk\to C[-1]\) (splitting of the counit) such that
\begin{gather}
\sum_{r+k+t=n} \xi_{r+1+t}(1^{\tens r}\tens\xi_k\tens1^{\tens t}) =0: C[-1]\to C[-1]^{\tens n}, \qquad \forall\, n\ge0,
\label{eq-xi-1xi1-0}
\\
\xi_2(1\tens\beps) =-1_{C[-1]}, \quad \xi_2(\beps\tens1) =1_{C[-1]}, \quad \xi_{a+1+c}(1^{\tens a}\tens\beps\tens1^{\tens c}) =0 \text{ \ if \ } a+c\ne1, \notag
%\label{eq-(1bone)b2-1-(bone1)b2-1}
\\
\bw \cdot \beps =1_\kk, \qquad \bw\xi_2 =-\bw\tens\bw. \notag
\end{gather}
\end{definition}

For any graded $\kk$\n-module $X$ its tensor algebra \(XT^\ge=\oplus_{n\ge0}X^{\tens n}\) is naturally embedded into its completed tensor algebra \(X\hat T^\ge=\prod_{n\ge0}X^{\tens n}\), \(\iota:XT^\ge\hookrightarrow X\hat T^\ge\).
An arbitrary $\iota$\n-derivation \(\xi:XT^\ge\to X\hat T^\ge\) is determined by its restriction to generators \(\check\xi:X\to X\hat T^\ge\).
In particular, the collection \((\xi_n)_{n\ge0}\) amounts to a degree 1 $\iota$\n-derivation \(\xi:C[-1]T^\ge\to C[-1]\hat T^\ge\) and equations~\eqref{eq-xi-1xi1-0} can be interpreted as $\xi^2=0$.

Getting rid of the shift $[-1]$ we rewrite the above via maps
\begin{gather*}
\delta_n =(-1)^n\sigma^{-1}\cdot\xi_n\cdot\sigma^{\tens n}: C\to C^{\tens n}, \qquad \deg\delta_n=2-n, \qquad n\ge0,
\\
\eps =\bigl( C \rTTo^{\sigma^{-1}} C[-1] \rTTo^\beps \kk \bigr), \qquad \deg \eps =0,
\\
\sfw =\bigl( \kk \rTTo^\bw C[-1] \rTTo^\sigma C \bigr), \qquad \deg \sfw =0.
\end{gather*}
In these terms \defref{def-strict-counit-complemented-curved-A8-coalgebra} becomes

\begin{definition}
A \emph{strict-counit-complemented curved \ainfm-coalgebra} \((C,(\delta_n)_{n\ge0},\eps,\sfw)\) consists of a graded $\kk$\n-module $C$, maps \(\delta_n:C\to C^{\tens n}\) of degree $2-n$ (cooperations) for $n\ge0$, a degree 0 map \(\eps:C\to\kk\) (strict counit) and a degree 0 map \(\sfw:\kk\to C\) (splitting of the counit) such that
\begin{gather}
\sum_{r+k+t=n} (-1)^{r+kt}\delta_{r+1+t}(1^{\tens r}\tens\delta_k\tens1^{\tens t}) =0: C\to C^{\tens n}, \qquad \forall\, n\ge0,
\label{eq-delta-1delta1-0}
\\
\delta_2(1\tens\eps) =1_C, \quad \delta_2(\eps\tens1) =1_C, \quad \delta_{a+1+c}(1^{\tens a}\tens\eps\tens1^{\tens c}) =0 \text{ \ if \ } a+c\ne1, \notag
%\label{eq-(1bone)b2-1-(bone1)b2-1}
\\
\sfw \cdot \eps =1_\kk, \qquad \sfw\delta_2 =\sfw\tens\sfw. \notag
\end{gather}
\end{definition}

Restricting the above notion and adding a conilpotency condition we get

\begin{definition}
A \emph{curved augmented coalgebra} \((C,\delta_2,\delta_1,\delta_0,\eps,\sfw)\) is a strict-\hspace{0pt}counit-\hspace{0pt}complemented curved \ainfm-coalgebra $C$ with $\delta_n=0$ for $n>2$ such that \((\bar{C}=\Ker\eps,\bar{\delta}_2)\) is conilpotent.
\end{definition}

For such coalgebra $C$ equations~\eqref{eq-delta-1delta1-0} reduce to the system
\begin{gather*}
\delta_2(1\tens\delta_2) =\delta_2(\delta_2\tens1), \qquad \delta_1\delta_2 =\delta_2(1\tens\delta_1 +\delta_1\tens1),
\\
\delta_1^2 =\delta_2(1\tens\delta_0 -\delta_0\tens1), \qquad \delta_1\delta_0 =0,
\\
\delta_2(1\tens\eps) =1_C, \qquad \delta_2(\eps\tens1) =1_C, \qquad \delta_1\eps =0, \qquad \sfw \cdot \eps =1_\kk, \qquad \sfw\delta_2 =\sfw\tens\sfw,
\end{gather*}
which tells that $C$ is a counital coassociative graded coalgebra \((C,\delta_2,\eps)\) with a degree~1 coderivation $\delta_1$, whose square is an inner coderivation determined by a functional  \(\delta_0:C\to\kk\) (curvature) of degree~2 and $\delta_1\delta_0=0$.
The degree~0 map \(\sfw:\kk\to C\) is a homomorphism of graded coalgebras, the augmentation of $C$.
In particular, \(\kk\sfw\hookrightarrow C\) is a direct complement to \(\bar{C}=\Ker\eps\).
The non-counital graded coalgebra $\bar{C}$ equipped with the comultiplication
\[ \bar{\delta}_2 =\delta_2 -1\tens\sfw -\sfw\tens1: \bar{C} \to \bar{C}\tens\bar{C}
\]
is conilpotent by assumption, that is,
\[ \bigcup_{n>1} \Ker(\bar\Delta^{(n)}: \bar{C} \to \bar{C}^{\tens n}) =\bar{C}.
\]

A morphism of curved \ainf-coalgebras $g:C\to D$ should be a $\dg$\n-algebra morphism \(g:C[-1]T^\ge\to D[-1]\hat T^\ge\), or, equivalently, a family of $\kk$\n-linear degree 0 maps \(g_n:C[-1]\to D[-1]^{\tens n}\), $n\ge0$, satisfying the equation \(g\xi=\xi g\).
However, to give sense to this equation in the form
\[ \sum_{r+k+t=n} g_{r+1+t}(1^{\tens r}\tens\xi_k\tens1^{\tens t}) =\sum_{i_1+\dots+i_k=n} \xi_k(g_{i_1}\tens g_{i_2}\tdt g_{i_k}): C[-1]\to D[-1]^{\tens n},
\]
one has to make additional assumptions.
We shall assume that $C$ is a curved coalgebra and $g_n$ vanish for $n>1$.
Moreover, we assume that $g_1$ preserves the splitting $\bw$:

\begin{definition}
A \emph{morphism of curved augmented coalgebras} \(g:C\to D\) is a pair \((g_1,g_0)\) consisting of $\kk$\n-linear maps \(g_1:C[-1]\to D[-1]\) and \(g_0:C[-1]\to\kk\) of degree~0 such that
\begin{gather*}
\xi^C_2(g_1\tens g_1) =g_1\xi^D_2, \qquad \xi^C_1g_1 +\xi^C_2(g_0\tens g_1 +g_1\tens g_0) =g_1\xi^D_1,
\\
\xi^C_0 +\xi^C_1g_0 +\xi^C_2(g_0\tens g_0) =g_1\xi^D_0, \qquad g_1\beps^D =\beps^C, \qquad \bw^Cg_1 =\bw^D.
\end{gather*}
The composition $h:C\to E$ of morphisms \(f:C\to D\) and \(g:D\to E\) is given by \(h_1=f_1g_1\), \(h_0=f_0+f_1g_0\).
\end{definition}

Rewriting this definition in terms of maps
\begin{alignat*}2
\sfg_1 &=\bigl( C \rTTo^{\sigma^{-1}} C[-1] \rTTo^{g_1} D[-1] \rTTo^\sigma D \bigr), &\qquad \deg \sfg_1 &=0,
\\
\sfg_0 &=\bigl( C \rTTo^{\sigma^{-1}} C[-1] \rTTo^{g_0} \kk \bigr), &\qquad \deg \sfg_0 &=1,
\end{alignat*}
we get

\begin{definition}\label{def-morphism-curved-augmented-coalgebras}
A \emph{morphism of curved augmented coalgebras} \(\sfg:C\to D\) is a pair \((\sfg_1,\sfg_0)\) consisting of $\kk$\n-linear maps \(\sfg_1:C\to D\) of degree~0 and \(\sfg_0:C\to\kk\) of degree~1 such that
\begin{gather*}
\delta^C_2(\sfg_1\tens\sfg_1) =\sfg_1\delta^D_2, \qquad \delta^C_1\sfg_1 +\delta^C_2(\sfg_0\tens\sfg_1 -\sfg_1\tens\sfg_0) =\sfg_1\delta^D_1,
\\
\delta^C_0 -\delta^C_1\sfg_0 -\delta^C_2(\sfg_0\tens\sfg_0) =\sfg_1\delta^D_0, \qquad \sfg_1\eps^D =\eps^C, \qquad \sfw^C\sfg_1 =\sfw^D.
\end{gather*}
The composition $\sfh:C\to E$ of morphisms \(\sff:C\to D\) and \(\sfg:D\to E\) is given by \(\sfh_1=\sff_1\sfg_1\), \(\sfh_0=\sff_0+\sff_1\sfg_0\).
The unit morphism is \((\id,0)\).
The category of curved augmented coalgebras is denoted $\CACoalg$.
\end{definition}

In particular, \(\sfg_1\) is a morphism of augmented graded coalgebras.
Actually, $\sfg_0$ occurs in the equations only as its restriction \(\sfg_0'=\sfg_0|_{\bar{C}}\) and validity of the equations does not depend on \(\und g=\sfw\sfg_0\in\kk^1\).
In fact, with the notation
\[ \bar\delta^C_2 =\bigl( \bar{C} \rMono C \rTTo^{\delta_2} C\tens C \rTTo^{\opr_C\tens\opr_C} \bar{C}\tens\bar{C} \bigr),
\]
we have
\[ \sfw\delta^C_2(\sfg_0\tens1 -1\tens\sfg_0) =(\sfw\sfg_0)\sfw -\sfw(\sfw\sfg_0) =0,
\]
which implies that
\begin{equation}
\begin{split}
\delta^C_2(\sfg_0\tens1 -1\tens\sfg_0) &=\opr_C(\bar\delta^C_2+1\tens\sfw+\sfw\tens1)(\sfg_0\tens1 -1\tens\sfg_0)
\\
&=\opr_C\bar\delta^C_2(\sfg_0\tens1 -1\tens\sfg_0).
\end{split}
\label{eq-delta-C2(g1-1g)-prC-delta-C2(g1-1g)}
\end{equation}
Since
\[ \sfw\delta^C_2(\sfg_0\tens\sfg_0) =(\sfw\sfg_0)^2 =0,
\]
we find that
\[ \delta^C_2(\sfg_0\tens\sfg_0) =\opr_C(\bar\delta^C_2+1\tens\sfw+\sfw\tens1)(\sfg_0\tens\sfg_0) =\opr_C\bar\delta^C_2(\sfg_0\tens\sfg_0).
\]
Thus \defref{def-morphism-curved-augmented-coalgebras} can be reformulated as follows.

\begin{definition}%\label{def-morphism-curved-augmented-coalgebras-2}
A \emph{morphism of curved augmented coalgebras} \(\sfg:C\to D\) is a triple \((\sfg_1,\sfg_0',\und g)\) consisting of a homomorphism of augmented graded coalgebras \(\sfg_1:C\to D\), a $\kk$\n-linear map \(\sfg_0':\bar C\to\kk\) of degree~1 and an element \(\und g\in\kk^1\) (of degree~1) such that
\begin{align*}
\delta^C_1\sfg_1 +\opr_C\bar\delta^C_2(\sfg_0'\tens\sfg_1 -\sfg_1\tens\sfg_0') &=\sfg_1\delta^D_1: C \to \bar D,
\\
\delta^C_0 -\delta^C_1\sfg_0' -\opr_C\bar\delta^C_2(\sfg_0'\tens\sfg_0') &=\sfg_1\delta^D_0: C \to \kk.
\end{align*}
The composition $\sfh:C\to E$ of morphisms \(\sff:C\to D\) and \(\sfg:D\to E\) is given by \(\sfh_1=\sff_1\sfg_1\), \(\sfh_0'=\sff_0'+\sff_1\sfg_0'\), \(\und h=\und f+\und g\).
The unit morphism is \((\id,0,0)\).
\end{definition}

\section{Bar and cobar constructions}
We are going to prove existence of two functors between categories of curved algebras and curved coalgebras, generalizing the well known bar and cobar constructions.

\subsection{Bar-construction.}
Let us construct a functor \(\Bbar:\UCCAlg\to\CACoalg\), the bar-\hspace{0pt}construction.
Let \(A=(A,(b_n)_{n\ge0},\bfeta,\bv)\) be a strict-unit-complemented curved \ainf-algebra.
The shift \(\bar{A}[1]\) of the $\kk$\n-submodule \(\bar{A}=\Ker\sfv\subset A\) is the image of an idempotent \(1-\bv\cdot\bfeta:A[1]\to A[1]\), which we write as the projection \(\opr=1-\bv\cdot\bfeta:A[1]\to{A}[1]\).
Define \(\Bbar A\) as \(\bar{A}[1]T^\ge\) equipped with the cut comultiplication \(\delta^{\Bbar A}_2\), the counit \(\eps^{\Bbar A}=\pr_0:\bar{A}[1]T^\ge\to\bar{A}[1]T^0=\kk\), the splitting \(\sfw^{\Bbar A}=\inj_0:\kk=\bar{A}[1]T^0\hookrightarrow\bar{A}[1]T^\ge\), the degree~1 coderivation \(\delta^{\Bbar A}_1=\bar{b}:\bar{A}[1]T^\ge\to\bar{A}[1]T^\ge\) given by its components
\[ \bar{b}_n =\bigl( \bar{A}[1]^{\tens n} \rMono A[1]^{\tens n} \rTTo^{b_n} A[1] \rTTo^\opr \bar{A}[1] \bigr), \qquad n\ge0,
\]
and a degree~2 functional
\[ \delta^{\Bbar A}_0 =-\bigl( \bar{A}[1]T^\ge \rMono A[1]T^\ge \rTTo^{\check{b}} A[1] \rTTo^\bv \kk \bigr).
\]
Clearly, \(\sfw^{\Bbar A}\) is a graded coalgebra homomorphism and the coalgebra \(\overline{\bar{A}[1]T^\ge}=\bar{A}[1]T^>\) with the cut comultiplication is conilpotent.

Let us verify the necessary identities.
Both sides of the equation
\[ (\delta^{\Bbar A}_1)^2 =\delta^{\Bbar A}_2(1\tens\delta^{\Bbar A}_0 -\delta^{\Bbar A}_0\tens1): \bar{A}[1]T^\ge \to \bar{A}[1]T^\ge
\]
are coderivations.
Hence, the equation is equivalent to its composition with \(\pr_1:\bar{A}[1]T^\ge\to\bar{A}[1]\).
That is, to
\[ \sum_{r+k+t=n} (1^{\tens r}\tens\bar b_k\tens1^{\tens t})\bar b_{r+1+t} =b_{n-1}\bv\tens1 -1\tens b_{n-1}\bv: \bar A[1]^{\tens n} \to \bar A[1]
\]
for all $n\ge0$.
This holds true due to computation
\begin{multline*}
\sum_{r+k+t=n} (1^{\tens r}\tens b_k(1-\bv\bfeta)\tens1^{\tens t})b_{r+1+t}\opr
\\
=-(1\tens b_{n-1}\bv\bfeta)b_2\opr -(b_{n-1}\bv\bfeta\tens1)b_2\opr =b_{n-1}\bv\tens1 -1\tens b_{n-1}\bv.
\end{multline*}

Furthermore,
\[ \delta^{\Bbar A}_1 \delta^{\Bbar A}_0 =-\bigl( \bar{A}[1]T^\ge \rTTo^{\bar b} \bar{A}[1]T^\ge \rTTo^{\check{b}} A[1] \rTTo^\bv \kk \bigr)
\]
vanishes due to
\begin{multline*}
-\sum_{r+k+t=n} (1^{\tens r}\tens b_k(1-\bv\bfeta)\tens1^{\tens t})b_{r+1+t}\bv
\\
=(1\tens b_{n-1}\bv\bfeta)b_2\bv +(b_{n-1}\bv\bfeta\tens1)b_2\bv  =\bv\tens b_{n-1}\bv -b_{n-1}\bv\tens\bv =0: \bar A[1]^{\tens n} \to \kk,
\end{multline*}
because \(\bar A[1]\bv=0\).
Thus the object $\Bbar A$ of $\CACoalg$ is well-defined.

Let us describe the functor \(\Bbar:\UCCAlg\to\CACoalg\) on morphisms.
It takes a morphism \(f=(f_1,f_0):A\to B\) to the morphism
\[ \sfBar f =\sfg =(\sfg_1,\sfg_0): \bar{A}[1]T^\ge \to \bar{B}[1]T^\ge,
\]
where the coalgebra homomorphism \(\sfBar_1f=\sfg_1=\bar{f}\) is specified by its components
\begin{equation}
\begin{split}
\bar f_1 &=\bigl( \bar{A}[1] \rMono A[1] \rTTo^{f_1} B[1] \rTTo^{1-\bv\cdot\bfeta} \bar{B}[1] \bigr),
\\
\bar f_0 &=\bigl( \kk \rTTo^{f_0} B[1] \rTTo^{1-\bv\cdot\bfeta} \bar{B}[1] \bigr) =0,
\end{split}
\label{eq-f1f0-A1k}
\end{equation}
and the degree~1 functional is
\begin{equation}
\sfBar_0f =\sfg_0 =\bigl( \bar{A}[1]T^\ge \rMono A[1]T^\ge \rTTo^{\check{f}} B[1] \rTTo^\bv \kk \bigr).
\label{eq-f0-g0-A1T}
\end{equation}
Notice that the coalgebra homomorphism $\bar{f}$ is strict, that is, it has only one non-vanishing component -- the first.
Thus, $\bar{f}$ preserves the number of tensor factors,
\[ \bar{f}| =\bar f_1^{\tens n}: \bar{A}[1]^{\tens n} \to \bar{B}[1]^{\tens n}, \qquad n\ge0.
\]
In particular, \(\sfw^{\Bbar A}\bar f=\sfw^{\Bbar B}\).

Let us check that $\sfg$ is indeed a morphism of $\CACoalg$.
It is required that
\[ \bar b^A\sfg_1 +\Delta(\sfg_0\tens\sfg_1 -\sfg_1\tens\sfg_0) =\sfg_1\bar b^B.
\]
All terms of this equation are $\bar{f}$\n-coderivations.
Hence, the equation follows from its composition with $\pr_1$:
\[ \bar b^A\check{\bar{f}} +\Delta(\sfg_0\tens\check{\bar{f}} -\check{\bar{f}}\tens\sfg_0) =\bar f\check{\bar{b}}^B: \bar{A}[1]T^\ge \to \bar{B}[1],
\]
that is, for all $n\ge0$
\begin{multline*}
\bar b_n\bar f_1 +f_0\bv\tens\bar f_n +f_1\bv\tens\bar f_{n-1} -\bar f_n\tens f_0\bv -\bar f_{n-1}\tens f_1\bv
\\
=\sum_{i_1+\dots+i_k=n} (\bar f_{i_1}\tens\bar f_{i_2}\tdt\bar f_{i_k})\bar b_k: \bar A[1]^{\tens n} \to \bar B[1].
\end{multline*}
In detail,
\begin{multline*}
b_n(1-\bv\bfeta)f_1\opr +\sum_{i_1+i_2=n} \bigl( f_{i_1}\bv\tens f_{i_2}\opr -f_{i_1}\opr\tens f_{i_2}\bv \bigr)
\\
=\sum_{i_1+\dots+i_k=n} [f_{i_1}(1-\bv\bfeta)\tdt f_{i_k}(1-\bv\bfeta)]b_k\opr.
\end{multline*}
Cancelling the summands without $\bv$ we reduce the equation to the valid identity
\begin{multline*}
\sum_{i_1+i_2=n} \bigl( f_{i_1}\bv\tens f_{i_2}\opr -f_{i_1}\opr\tens f_{i_2}\bv \bigr) =-\sum_{i_1+i_2=n} \bigl[ (f_{i_1}\bv\bfeta\tens f_{i_2})b_2\opr +(f_{i_1}\tens f_{i_2}\bv\bfeta)b_2\opr \bigr].
\end{multline*}

Another equation to prove,
\[ \check b^A\bv +\bar b^A\check f\bv +\Delta(\check{f}\bv\tens\check{f}\bv) =\bar f\check{b}^B\bv:  \bar{A}[1]T^\ge \to \kk,
\]
is written explicitly as
\begin{multline*}
b_n\bv +b_n(1-\bv\bfeta)f_1\bv +\sum_{i_1+i_2=n} f_{i_1}\bv\tens f_{i_2}\bv 
\\
=\sum_{i_1+\dots+i_k=n} [f_{i_1}(1-\bv\bfeta)\tdt f_{i_k}(1-\bv\bfeta)]b_k\bv: \bar{A}[1]^{\tens n} \to \kk.
\end{multline*}
Cancelling the first and the third summands as well as summands that contain $\bv$ only at the end, we obtain the valid equation
\begin{multline*}
\sum_{i_1+i_2=n} f_{i_1}\bv\tens f_{i_2}\bv =-\sum_{i_1+i_2=n} \bigl[ (f_{i_1}\bv\bfeta\tens f_{i_2})b_2\bv +(f_{i_1}\tens f_{i_2}\bv\bfeta)b_2\bv +(f_{i_1}\bv\tens f_{i_2}\bv)\bfeta\bv \bigr].
\end{multline*}

The identity morphism $f=(\id,0)$ is mapped to the identity morphism $\sfBar f=(\id,0)$.
Let us verify that $\Bbar$ agrees with the composition.
If $h=fg$ in $\UCCAlg$, \(h_1=f_1g_1\), \(h_0=g_0+f_0g_1\), then $\bar h=\bar f\bar g$.
In fact, the equation
\[ \sum_{i_1+\dots+i_k=n} (\bar f_{i_1}\tens\bar f_{i_2}\tdt\bar f_{i_k})\bar g_k =\bar h_n
\]
has the only non-vanishing realization $\bar f_1\bar g_1=\bar h_1$.
Furthermore,
\[ \sfBar_0f +(\sfBar_1f) \cdot \sfBar_0g =\sfBar_0h
\]
since
\[ \check f\bv +\bar f\check g\bv =\check{h}\bv: \bar{A}[1]T^\ge \to \kk.
\]
In fact, in arity $n$ the left hand side is
\[ f_n\bv +f_n(1-\bv\bfeta)g_1\bv +\delta_{n,0}g_0\bv =(f_ng_1 +\delta_{n,0}g_0)\bv =h_n\bv.
\]
The functor \(\Bbar:\UCCAlg\to\CACoalg\) (the bar-construction) is described.

\subsection{Cobar-construction.}
Let us construct a functor \(\Cobar:\CACoalg\to\UCCAlg\), the cobar-\hspace{0pt}construction.
Let \(C=(C,(\xi_n)_{n\ge0},\beps,\bw)\) be a strict-counit-complemented curved \ainf-coalgebra.
The shift \(\bar{C}[-1]\) of the $\kk$\n-submodule \(\bar{C}=\Ker\eps\subset C\) is the image of an idempotent \(1-\beps\cdot\bw:C[-1]\to C[-1]\), which we write as the projection \(\opr=1-\beps\cdot\bw:C[-1]\to\bar{C}[-1]\).
Define \(\Cobar C\) as \(\bar{C}[-1]T^\ge\) equipped with the multiplication \(m^{\Cobar C}_2\) in the tensor algebra, the unit
\(\eta^{\Cobar C}=\inj_0:\kk=\bar{C}[-1]T^0\hookrightarrow\bar{C}[-1]T^\ge\), the splitting \(\sfv^{\Cobar C}=\pr_0:\bar{C}[-1]T^\ge\to\bar{C}[-1]T^0=\kk\), the degree~1 derivation \(m^{\Cobar C}_1=\bar\xi:\bar{C}[-1]T^\ge\to\bar{C}[-1]T^\ge\) given by its components
\[ \bar\xi_n =\bigl( \bar{C}[-1] \rMono C[-1] \rTTo^{\xi_n} C[-1]^{\tens n} \rTTo^{\opr^{\tens n}} \bar{C}[-1]^{\tens n} \bigr), \qquad n\ge0,
\]
and a degree~2 element
\[ m^{\Cobar C}_0 =-\bw\tens\bw -\sum_{n\ge0}\bw\xi_n \in \bar{C}[-1]\hat T^\ge.
\]
For general curved \ainf-coalgebra $C$ the element \(m^{\Cobar C}_0\) does not belong to \(\bar{C}[-1]T^\ge\), however, if $C$ is a curved augmented coalgebra, then it does.
Conilpotency of $\bar{C}$ is not needed for existence of $\Cobar C$.
Let us verify necessary identities.

If $n\ne2$, then \(\bar\xi_n=\xi_n\big|_{\bar{C}[-1]}\).
Furthermore,
\[ \bar\xi_2 =\xi_2\big|_{\bar{C}[-1]}\cdot [(1-\beps\bw)\tens(1-\beps\bw)] =\xi_2\big|_{\bar{C}[-1]} +1\tens\bw -\bw\tens1.
\]
Extension of this map satisfies
\begin{equation}
\bar\xi_2 =\xi_2[(1-\beps\bw)\tens(1-\beps\bw)] =\xi_2 +1\tens\bw -\bw\tens1 -\beps(\bw\tens\bw): C[-1] \to C[-1]^{\tens2}.
\label{eq-xi2-xi2-1ww1eww}
\end{equation}
Both sides of the equation
\[ (m^{\Cobar C}_1)^2 =(m^{\Cobar C}_0\tens1 -1\tens m^{\Cobar C}_0) m^{\Cobar C}_2
\]
are derivations.
It is equivalent to its restriction to generators $\bar{C}[-1]$:
\begin{equation}
\sum_{r+k+t=n}\bar\xi_{r+1+t}(1^{\tens r}\tens\bar\xi_k\tens1^{\tens t}) =(m^{\Cobar C}_0)_{n-1}\tens1 -1\tens(m^{\Cobar C}_0)_{n-1}: \bar C[-1]\to\bar C[-1]^{\tens n}.
\label{eq-xi-xi-m0-m0}
\end{equation}
Let us prove this for
\[ (m^{\Cobar C}_0)_2 =-\bw\tens\bw -\bw\xi_2 =0, \qquad (m^{\Cobar C}_0)_{n-1} =-\bw\xi_{n-1} \quad \text{if } n\ne3.
\]
In fact, \eqref{eq-xi-xi-m0-m0} is obvious for $n=0$.
It says for $n=1$ that
\[ \xi_1^2 +\bar\xi_2(1\tens\xi_0 +\xi_0\tens1) =(1\tens\bw -\bw\tens1) (1\tens\xi_0 +\xi_0\tens1) =(\bw\xi_0) -\xi_0\bw +\xi_0\bw -(\bw\xi_0) =0
\]
as it has to be.
If $n=2$ or $n\ge4$, then the left hand side of \eqref{eq-xi-xi-m0-m0} is
\begin{align*}
&\xi_1\xi_n +\bar\xi_2(\xi_{n-1}\tens1 +1\tens\xi_{n-1}) +\dots+\xi_{n-1}\sum_{r+2+t=n}1^{\tens r}\tens\bar\xi_2\tens1^{\tens t} +\dots
\\
&=(1\tens\bw -\bw\tens1)(\xi_{n-1}\tens1 +1\tens\xi_{n-1})
\\
&\hspace*{7em} +\xi_{n-1}\sum_{r+2+t=n}(1^{\tens(r+1)}\tens\bw\tens1^{\tens t} -1^{\tens r}\tens\bw\tens1^{\tens(1+t)})
\\
&=-\xi_{n-1}\tens\bw +1\tens\bw\xi_{n-1} -\bw\xi_{n-1}\tens1 -\bw\tens\xi_{n-1} +\xi_{n-1}(1^{\tens(n-1)}\tens\bw -\bw\tens1^{\tens(n-1)})
\\
&=1\tens\bw\xi_{n-1} -\bw\xi_{n-1}\tens1,
\end{align*}
as claimed.
If $n=3$, then the left hand side of \eqref{eq-xi-xi-m0-m0} is
\begin{align*}
&\xi_1\xi_3 +\bar\xi_2(\bar\xi_2\tens1 +1\tens\bar\xi_2) +\dots
\\
&=(\xi_2 +1\tens\bw -\bw\tens1)[(\xi_2 -\bw\tens1)\tens1 +1\tens(\xi_2 +1\tens\bw)] -\xi_2(\xi_2\tens1 +1\tens\xi_2)
\\
&=1\tens(\bw\xi_2 +\bw\tens\bw) -(\bw\xi_2 +\bw\tens\bw)\tens1 =0,
\end{align*}
as claimed.

The expression \(m^{\Cobar C}_0m^{\Cobar C}_1\) is a well-defined element of \(\bar{C}[-1]\hat T^\ge\).
Its $n$\n-th component is
\begin{align}
&m^{\Cobar C}_0m^{\Cobar C}_1\pr_n
=-\bw\tens\bw\bar\xi_{n-1} +\bw\bar\xi_{n-1}\tens\bw -\sum_{r+k+t=n}\bw\xi_{r+1+t}(1^{\tens r}\tens\bar\xi_k\tens1^{\tens t}) \notag
\\
&=-\bw\tens\bw\bar\xi_{n-1} +\bw\bar\xi_{n-1}\tens\bw -\bw\xi_{n-1}\sum_{r+2+t=n}(1^{\tens(r+1)}\tens\bw\tens1^{\tens t} -1^{\tens r}\tens\bw\tens1^{\tens(1+t)}) \notag
\\
&=-\bw\tens\bw\bar\xi_{n-1} +\bw\bar\xi_{n-1}\tens\bw -\bw\xi_{n-1}(1^{\tens(n-1)}\tens\bw -\bw\tens1^{\tens(n-1)}).
\label{eq-m0m1-wwwwwwwwwwww}
\end{align}
If $n\ne3$, then \(\bar\xi_{n-1}=\xi_{n-1}\) and the obtained expression equals
\[ -\bw\tens\bw\xi_{n-1} +\bw\xi_{n-1}\tens\bw -\bw\xi_{n-1}\tens\bw +\bw\tens\bw\xi_{n-1} =0.
\]
If $n=3$, then \eqref{eq-m0m1-wwwwwwwwwwww} equals
\[ -\bw\tens[\bw(1\tens\bw -\bw\tens1)] +[\bw(1\tens\bw -\bw\tens1)]\tens\bw =\bw\tens\bw\tens\bw(-1-1+1+1) =0.
\]
Thus \(m^{\Cobar C}_0m^{\Cobar C}_1=0\).
We obtain a map \(\Ob\CACoalg\to\Ob\UCCAlg\).

Let us describe the functor \(\Cobar:\CACoalg\to\UCCAlg\) on morphisms.
It takes a morphism \(g=(g_1,g_0):C\to D\) to the morphism
\[ \sfCobar g =\sff =(\sff_1,\sff_0): \bar{C}[-1]T^\ge \to \bar{D}[-1]T^\ge,
\]
where the algebra homomorphism \(\sfCobar_1g=\sff_1=\bar{g}\) is specified by its components
\begin{equation}
\begin{split}
\bar g_1 &=g_1 =\bigl( \bar{C}[-1] \rMono C[-1] \rTTo^{g_1} D[-1] \rTTo^\opr \bar{D}[-1] \bigr),
\\
\bar g_0 &=g_0'=\bigl( \bar{C}[-1] \rMono C[-1] \rTTo^{g_0} \kk \bigr),
\label{eq-g1g0g0}
\end{split}
\end{equation}
and the degree~1 element is
\begin{equation*}
\sfCobar_0g =\sff_0 =\bigl( \kk \rTTo^\bw C[-1] \rTTo^{\check{g}} D[-1]T^\ge \rTTo^{\opr T^\ge} \bar{D}[-1]T^\ge \bigr),
\end{equation*}
which we write as \(\bw\check{\bar{g}}\) extending the notation.
This element has the only non-vanishing component
\begin{align*}
\sff_{00} &=\bigl( \kk \rTTo^\bw C[-1] \rTTo^{g_0} \kk \bigr) =\bw g_0.
\\
\intertext{In fact,}
\sff_{01} &=\bigl( \kk \rTTo^\bw C[-1] \rTTo^{g_1} D[-1] \rTTo^\opr \bar{D}[-1] \bigr) =\bw g_1\opr =0.
\end{align*}
Thus,
\begin{align}
\sfCobar_0g =\sff_0 &=\bigl( \kk \rTTo^\bw C[-1] \rTTo^{g_0} \kk \rMono^{\inj_0} \bar{D}[-1]T^\ge \bigr), \notag
\\
\und{\Cobar g} =\und f &=\bigl( \kk \rTTo^\sfw C \rTTo^{\sfg_0} \kk \bigr) =\und g.
\label{eq-und-Cobar-g-und-f}
\end{align}

Let us check that $\sff$ is indeed a morphism of $\UCCAlg$.
It is required that
\[ \bar g\bar\xi +(\bar g\tens\sff_0 -\sff_0\tens\bar g)m_2 =\bar\xi\bar g.
\]
The second term vanishes, but this form of equation is easier to deal with.
All terms of this equation are $\bar{g}$\n-derivations.
Hence, the equation is equivalent to its restriction to $\bar{C}[-1]$:
\[ \check{\bar{g}}\bar\xi +(\check{\bar{g}}\tens\sff_0 -\sff_0\tens\check{\bar{g}})m_2 =\check{\bar{\xi}}\bar g: \bar{C}[-1] \to \bar{D}[-1]T^\ge,
\]
which means that for all $n\ge0$
\begin{multline*}
\bar g_1\bar\xi_n +(\bar g_n\tens\bw g_0 +\bar g_{n-1}\tens\bw g_1\opr -\bw g_0\tens\bar g_n -\bw g_1\opr\tens\bar g_{n-1})m_2
\\
=\sum_{i_1+\dots+i_k=n} \bar\xi_k(\bar g_{i_1}\tens\bar g_{i_2}\tdt\bar g_{i_k}): \bar C[-1] \to \bar D[-1]^{\tens n}.
\end{multline*}
When written explicitly,
\begin{multline*}
g_1(1-\beps\bw)\xi_n\opr^{\tens n}
\\
+(g_n\opr^{\tens n}\tens\bw g_0 +g_{n-1}\opr^{\tens(n-1)}\tens\bw g_1\opr -\bw g_0\tens g_n\opr^{\tens n} -\bw g_1\opr\tens g_{n-1}\opr^{\tens(n-1)})m_2
\\
=\sum_{i_1+\dots+i_k=n} \xi_k(g_{i_1}\tdt g_{i_k})\opr^{\tens n} +\sum_{i_1+i_2=n} \bigl( g_{i_1}\tens\bw g_{i_2} -\bw g_{i_1}\tens g_{i_2} \bigr)\opr^{\tens n}: \bar C[-1] \to \bar D[-1]^{\tens n},
\end{multline*}
it becomes obvious.

Another equation must hold,
\[ -\bw\tens\bw -\bw\check\xi -\bw\check{\bar{g}}\bar\xi -(\bw\check{\bar{g}}\tens\bw\check{\bar{g}})m_2 =-\bw\check{\bar{g}}\tens\bw\check{\bar{g}} -\sum_{k\ge0}\bw\xi_k\check{\bar{g}}^{\tens k}: \kk\to\bar{D}[-1]T^\ge.
\]
After cancelling two summands and changing the sign the equation is written as
\begin{equation*}
\delta_{n,2}\bw\tens\bw +\bw\xi_n +\bw\bar{g}_1\bar\xi_n =\sum_{i_1+\dots+i_k=n} \bw\xi_k(\bar g_{i_1}\tens\bar g_{i_2}\tdt\bar g_{i_k}): \kk \to \bar D[-1]^{\tens n}.
\end{equation*}
Explicitly:
\begin{equation*}
\delta_{n,2}\bw\tens\bw +\bw\xi_n +\bw g_1(1-\beps\bw)\bar\xi_n =\sum_{i_1+\dots+i_k=n} \bw\xi_k(g_{i_1}\tens g_{i_2}\tdt g_{i_k})\opr^{\tens n}.
\end{equation*}
Cancelling $\bw g_1\bar\xi_n$ against the right hand side we come to the valid equation
\begin{equation*}
\delta_{n,2}\bw\tens\bw +\bw\xi_n -\bw\bar\xi_n =0: \kk \to \bar D[-1]^{\tens n}.
\end{equation*}
In fact, \(\bar\xi_n=\xi_n\) for $n\ne2$ and the equation is obvious.
For $n=2$ we have by \eqref{eq-xi2-xi2-1ww1eww}
\[ \bw\tens\bw +\bw\xi_2 -\bw\xi_2 -\bw(1\tens\bw) +\bw(\bw\tens1) +\bw\beps(\bw\tens\bw) =0.
\]

The identity morphism $g=(\id,0)$ is mapped to the identity morphism $\sfCobar g=(\id,0)$.
Let us verify that $\Cobar$ agrees with the composition.
If $h=\bigl(C\rto fD\rto gE\bigr)$ in $\CACoalg$, \(h_1=f_1g_1\), \(h_0=f_0+f_1g_0\), then
\[ (\sfCobar_1f)\cdot\sfCobar_1g =\bar f\bar g =\bar h =\sfCobar_1h.
\]
In fact, the equation
\[ \sum_{i_1+\dots+i_k=n} \bar f_k(\bar g_{i_1}\tens\bar g_{i_2}\tdt\bar g_{i_k}) =\bar h_n: \bar C[-1] \to \bar E[-1]^{\tens n}
\]
for $n=1$ holds due to
\[ \bar f_1\bar g_1 =f_1(1-\beps\bw)g_1\opr =f_1g_1\opr =h_1\opr =\bar{h}_1: \bar C[-1] \to \bar E[-1],
\]
and for $n=0$ it holds due to
\[ \bar f_0 +\bar f_1\bar g_0 =f_0 +f_1(1-\beps\bw)g_0 =f_0+f_1g_0 =h_0 =\bar{h}_0: \bar C[-1] \to \kk.
\]

Furthermore,
\[ \sfCobar_0g +(\sfCobar_0f) \cdot \sfCobar_1g =\sfCobar_0h
\]
since
\[ \bw\check{\bar{g}} +\bw\check{\bar{f}}\bar g=\bw\check{\bar{h}}: \kk \to \bar{E}[-1]T^\ge.
\]
In fact, the $n$\n-th component of the left hand side is
\[ \bw\bar{g}_n +\sum_{i_1+\dots+i_k=n} \bw\bar f_k(\bar g_{i_1}\tens\bar g_{i_2}\tdt\bar g_{i_k}): \kk\to \bar E[-1]^{\tens n},
\]
which for $n=1$ transforms to
\[ \bw\bar g_1 +\bw\bar f_1\bar g_1 = \bw\bar g_1 +\bw f_1(1-\beps\bw)\bar g_1 =\bw f_1g_1\opr =\bw\bar{h}_1: \kk\to \bar E[-1],
\]
and for $n=0$ equals
\[ \bw\bar g_0 +\bw\bar f_0 +\bw\bar f_1\bar g_0 =\bw\bar g_0 +\bw\bar f_0 +\bw f_1(1-\beps\bw)\bar g_0 =\bw(f_0+f_1g_0)\opr =\bw\bar{h}_0.
\]
The functor \(\Cobar:\CACoalg\to\UCCAlg\) (the cobar-construction) is described.

\section{Adjunction}
We are showing that the two (bar and cobar) constructions are functors, adjoint to each other.
The adjunction bijection will be the top row of the following diagram.
The two middle rows are natural transformations defined so that the two lower squares commute:
\[
\begin{diagram}[inline]%[LaTeXeqno]
\UCCAlg(\bar{C}[-1]T^\ge,A) &\rDashTo &\CACoalg(C,\bar A[1]T^\ge)
\\
\dMono &&\dMono
\\
\gr\alg(\bar{C}[-1]T^\ge,A)\times\gr(\kk[-1],\kk) &\rTTo_\sim &\gr\text-\nuCoalg(\bar{C},\bar A[1]T^>) \times\gr(C,\kk[1])
\\
\dTTo<\wr &&\dTTo<\wr>{\_\cdot\pr_1\times\id}
\\
\gr(\bar{C}[-1],A)\times\gr(\kk[-1],\kk) &\rTTo_\sim &\gr(\bar{C},\bar A[1])\times\gr(C,\kk[1])
\\
\dEq && \dEq
\\
\gr(\bar{C}[-1],\bar A)\times\gr(\bar{C}[-1],\kk)\times\gr(\kk[-1],\kk) &\rTTo_\sim^{[1]} &\gr(\bar C,\bar A[1])\times\gr(\bar C,\kk[1])\times\gr(\kk,\kk[1])
%\label{dia-C[-1]T-O}
\end{diagram}
\]
Notice that the set of morphisms of augmented graded coalgebras \(C\to\bar A[1]T^\ge\) is in bijection with the set of morphisms of graded non-counital coalgebras \(\gr\text-\nuCoalg(\bar{C},\bar A[1]T^>)\).
The functor \(X\mapsto XT^>=\oplus_{n>0}X^{\tens n}\) has the structure of a comonad and $T^>$\n-coalgebras are precisely conilpotent non-counital coalgebras \cite[Section~6.7]{BesLyuMan-book}.
Since $\bar C$ is conilpotent, the arrow $\_\cdot\pr_1\oplus\id$ is a bijection by the well known lemma on Kleisli categories (generalized to multicategories in \cite[Lemma~5.3]{BesLyuMan-book}).
Thus the second horizontal map is a bijection as well.
Morphisms \(\sff:\bar{C}[-1]T^\ge\to A\in\UCCAlg\) and \(\sfg:C\to\bar A[1]T^\ge\in\CACoalg\) are related as the following diagram shows.
It consists of elements (morphisms of degree~0) of vertices of the previous diagram.
For instance, \(g^1_1=\check\sfg_1=(\check\sff_1\opr)[1]=f_1^1[1]\), etc.
Equivalently,
\begin{subequations}
\begin{alignat}2
\sigma^{-1}\check\sff_1\opr &=\check\sfg_1\sigma^{-1} &: \bar C &\to \bar A,
\label{eq-sfp-gs}
\\
\sigma^{-1}\check\sff_1\sfv &=\sfg_0|_{\bar{C}} &: \bar C &\to \kk,
\label{eq-sfv-g}
\\
\und f &=\sfw\sfg_0 &: \kk &\to \kk,
\label{eq-f-wg}
\end{alignat}
\end{subequations}
where all components are listed in
\begin{diagram}[LaTeXeqno]
\sff &\rMapsTo &\sfg
\\
\dMapsTo &&\dMapsTo
\\
(\sff_1,\sigma\und f) &\rMapsTo &(\sfg_1,\sfg_0\sigma)
\\
\dMapsTo &&\dMapsTo
\\
(\check\sff_1,\sigma\und f) &\rMapsTo &(\check\sfg_1,\sfg_0\sigma)
\\
\dMapsTo &&\dMapsTo
\\
(\check\sff_1\opr,\check\sff_1\sfv,\sigma\und f) &\rMapsTo^{[1]} &(\check\sfg_1,\sfg_0|_{\bar{C}}\sigma,\sfw\sfg_0\sigma)
\label{dia-f-g-fsf-ggs}
\end{diagram}

We are going to show that systems of equations on pairs \((\sff_1,\sigma\und f)\) and \((\sfg_1,\sfg_0\sigma)\) saying that these pairs are morphisms of $\UCCAlg$ and $\CACoalg$ are equivalent.
In fact, these systems are
\begin{equation}
\begin{split}
\check\sff_1m^A_1 &=\check m^{\Cobar C}_1\sff_1: \bar{C}[-1]\to A,
\\
m^A_0 &=m^{\Cobar C}_0\sff_1: \kk \to A,
\end{split}
\label{eq-fm-mf-m-mf}
\end{equation}
\begin{equation}
\begin{split}
\delta^C_1\check\sfg_1 +\delta^C_2(\sfg_0\tens1 -1\tens\sfg_0)\check\sfg_1 &=\sfg_1\check\delta^{\Bbar A}_1: C\to \bar A[1],
\\
\delta^C_0 -\delta^C_1\sfg_0 -\delta^C_2(\sfg_0\tens\sfg_0) &=\sfg_1\delta^{\Bbar A}_0: C\to \kk.
\end{split}
\label{eq-dg-dg11gg-gd-d-dg}
\end{equation}
Note that the image of any coderivation $C\to C$ is contained in $\bar{C}=\Ker\eps$.
In more detail equations~\eqref{eq-fm-mf-m-mf} and \eqref{eq-dg-dg11gg-gd-d-dg} read
\begin{equation}
\begin{split}
\check\sff_1m^A_1 &= \xi_0\eta +\xi_1\check\sff_1 +\bar\xi_2(\check\sff_1\tens\check\sff_1)m^A_2: \bar{C}[-1]\to A,
\\
m^A_0 &= -\bw\xi_0\eta^A -\bw\xi_1\check\sff_1: \kk \to A,
\end{split}
\label{eq-fm-xe-xf-xffm}
\end{equation}
\begin{gather}
\delta^C_1\check\sfg_1 +\delta^C_2(\sfg_0\tens1 -1\tens\sfg_0)\check\sfg_1 \hspace*{18em}\notag
\\
=\eps^Cb^A_0\opr_A +\opr_C\check\sfg_1b^A_1\opr_A +\opr_C\bar\delta^C_2(\check\sfg_1\tens\check\sfg_1)b^A_2\opr_A: C\to \bar A[1],
\label{eq-dg-dg11gg-ebp-pgbp-dpgpgbp}
\\
\delta^C_0 -\delta^C_1\sfg_0 -\delta^C_2(\sfg_0\tens\sfg_0) =-\eps^Cb^A_0\bv -\opr_C\check\sfg_1b^A_1\bv -\opr_C\bar\delta^C_2(\check\sfg_1\tens\check\sfg_1)b^A_2\bv: C\to \kk. \notag
\end{gather}
Let us rewrite systems \eqref{eq-fm-xe-xf-xffm} and \eqref{eq-dg-dg11gg-ebp-pgbp-dpgpgbp} splitting each equation in two accordingly to splitting the target $A$ or the source $C$ in two summands
\begin{subequations}
\begin{alignat}2
\check\sff_1m^A_1\opr_A &= \xi_1\check\sff_1\opr_A +\bar\xi_2(\check\sff_1\tens\check\sff_1)m^A_2\opr_A:\, & \bar{C}[-1] &\to \bar A,
\\
\check\sff_1m^A_1\sfv &= \xi_0 +\xi_1\check\sff_1\sfv +\bar\xi_2(\check\sff_1\tens\check\sff_1)m^A_2\sfv:\, & \bar{C}[-1] &\to \kk,
\\
m^A_0\opr_A &= -\bw\xi_1\check\sff_1\opr_A: & \kk &\to \bar A,
\\
m^A_0\sfv &= -\bw\xi_0 -\bw\xi_1\check\sff_1\sfv: & \kk &\to \kk,
\end{alignat}
\label{eq-fmp-fmv-mp-mv}
\end{subequations}
\begin{subequations}
\begin{alignat}2
\delta^C_1\check\sfg_1 +\delta^C_2(\sfg_0\tens1 -1\tens\sfg_0)\check\sfg_1 &= \check\sfg_1b^A_1\opr_A +\bar\delta^C_2(\check\sfg_1\tens\check\sfg_1)b^A_2\opr_A:\, & \bar C &\to \bar A[1], 
\label{eq-dg-d-wd-wd-a}
\\
\delta^C_0 -\delta^C_1\sfg_0 -\bar\delta^C_2(\sfg_0\tens\sfg_0) &= -\check\sfg_1b^A_1\bv -\bar\delta^C_2(\check\sfg_1\tens\check\sfg_1)b^A_2\bv: & \bar C &\to \kk,
\\
\sfw\delta^C_1\check\sfg_1 &= b^A_0\opr_A: & \kk &\to \bar A[1],
\\
\sfw\delta^C_0 -\sfw\delta^C_1\sfg_0 &= -b^A_0\bv: & \kk &\to \kk.
\end{alignat}
\label{eq-dg-d-wd-wd}
\end{subequations}

We claim that equations (\ref{eq-fmp-fmv-mp-mv}x) and (\ref{eq-dg-d-wd-wd}x) are equivalent for x$\,\in\{$a,b,c,d$\}$.
In fact, let us rewrite the equations once again in the same order replacing $m$ and $\delta$ with their definitions and composing with $\sigma^{-1}$ wherever appropriate:
\begin{subequations}
\begin{alignat}2
\sigma^{-1}\check\sff_1\sigma b^A_1\sigma^{-1}\opr_A +\sigma^{-1}\xi_1\check\sff_1\opr_A +\sigma^{-1}\bar\xi_2(\check\sff_1\sigma\tens\check\sff_1\sigma)b^A_2\sigma^{-1}\opr_A &=0 :\, & \bar{C} &\to \bar A,
\\
\sigma^{-1}\check\sff_1\sigma b^A_1\sigma^{-1}\sfv +\sigma^{-1}\xi_0 +\sigma^{-1}\xi_1\check\sff_1\sfv +\sigma^{-1}\bar\xi_2(\check\sff_1\sigma\tens\check\sff_1\sigma)b^A_2\sigma^{-1}\sfv &=0 :\, & \bar{C} &\to \kk,
\\
b^A_0\sigma^{-1}\opr_A +\bw\xi_1\check\sff_1\opr_A &=0 : & \kk &\to \bar A,
\\
b^A_0\sigma^{-1}\sfv +\bw\xi_0 +\bw\xi_1\check\sff_1\sfv &=0 : & \kk &\to \kk.
\end{alignat}
\label{eq-sfsbsp-sfsbsv-bsp-bsv}
\end{subequations}
In transforming \eqref{eq-dg-d-wd-wd-a} use that
\[ \delta^C_2(\sfg_0\tens1 -1\tens\sfg_0) =\delta^C_2(\opr_C\tens\opr_C)(\sfg_0\tens1 -1\tens\sfg_0): \bar{C} \to C
\]
actually takes values in $\bar{C}$, see \eqref{eq-delta-C2(g1-1g)-prC-delta-C2(g1-1g)}.
The second system is
\begin{subequations}
\begin{alignat}2
\sigma^{-1}\xi^C_1\sigma\check\sfg_1\sigma^{-1} +\sigma^{-1}\bar\xi^C_2(\sigma\sfg_0\tens\sigma +\sigma\tens\sigma\sfg_0)\check\sfg_1\sigma^{-1} \hspace*{5em} &&& \notag
\\
+\check\sfg_1b^A_1\opr_A\sigma^{-1} +\sigma^{-1}\bar\xi^C_2(\sigma\check\sfg_1\tens\sigma\check\sfg_1)b^A_2\opr_A\sigma^{-1} &=0 :\, & \bar C &\to \bar A,
\\
\sigma^{-1}\xi^C_0 +\sigma^{-1}\xi^C_1\sigma\sfg_0 +\sigma^{-1}\bar\xi^C_2(\sigma\sfg_0\tens\sigma\sfg_0) +\check\sfg_1b^A_1\bv \hspace*{5em} &&& \notag
\\
+\sigma^{-1}\bar\xi^C_2(\sigma\check\sfg_1\tens\sigma\check\sfg_1)b^A_2\bv &=0 : & \bar C &\to \kk,
\\
\sfw\sigma^{-1}\xi^C_1\sigma\check\sfg_1\sigma^{-1} +b^A_0\opr_A\sigma^{-1} &=0 : & \kk &\to \bar A,
\\
\sfw\sigma^{-1}\xi^C_0 +\sfw\sigma^{-1}\xi^C_1\sigma\sfg_0 +b^A_0\bv &=0 : & \kk &\to \kk.
\end{alignat}
\label{eq-sxsgs-sx-wsxs-wsx}
\end{subequations}
Accordingly to our system of notation \(\sigma\opr=\opr\sigma\).
We shall use that \(\check\sff_1=\check\sff_1\opr_A+\check\sff_1\sfv\eta\).
Substituting relations \eqref{eq-sfp-gs} and \eqref{eq-sfv-g} into the above equations we find that the latter are pairwise equivalent.

In fact, (\ref{eq-sfsbsp-sfsbsv-bsp-bsv}c) is equivalent to (\ref{eq-sxsgs-sx-wsxs-wsx}c) and (\ref{eq-sfsbsp-sfsbsv-bsp-bsv}d) is equivalent to (\ref{eq-sxsgs-sx-wsxs-wsx}d).
Equivalence of (\ref{eq-sfsbsp-sfsbsv-bsp-bsv}a) and (\ref{eq-sxsgs-sx-wsxs-wsx}a) follows from the identity
\begin{equation*}
\sigma^{-1}\bar\xi_2(\check\sff_1\sfv\eta\tens\check\sff_1\opr_A +\check\sff_1\opr_A\tens\check\sff_1\sfv\eta)m^A_2\opr_A =\sigma^{-1}\bar\xi_2(\sigma\sfg_0|_{\bar C}\tens\sigma +\sigma\tens\sigma\sfg_0|_{\bar C})\check\sfg_1\sigma^{-1}. 
\end{equation*}
Equivalence of (\ref{eq-sfsbsp-sfsbsv-bsp-bsv}b) and (\ref{eq-sxsgs-sx-wsxs-wsx}b) follows from the identity
\begin{multline*}
\sigma^{-1}\bar\xi^C_2[(\check\sff_1\opr_A +\check\sff_1\sfv\eta)\tens(\check\sff_1\opr_A +\check\sff_1\sfv\eta)]m^A_2\sfv
\\
=\sigma^{-1}\bar\xi^C_2(\sigma\sfg_0\tens\sigma\sfg_0) +\sigma^{-1}\bar\xi^C_2(\sigma\check\sfg_1\tens\sigma\check\sfg_1) b^A_2\sigma^{-1}\sfv: \bar{C} \to \kk,
\end{multline*}
which can be expanded to
\begin{equation*}
\sigma^{-1}\bar\xi^C_2(\check\sff_1\opr_A\tens\check\sff_1\sfv\eta +\check\sff_1\sfv\eta\tens\check\sff_1\opr_A +\check\sff_1\sfv\eta\tens\check\sff_1\sfv\eta)]m^A_2\sfv =\sigma^{-1}\bar\xi^C_2(\sigma\sfg_0|_{\bar{C}}\tens\sigma\sfg_0|_{\bar{C}}): \bar{C} \to \kk.
\end{equation*}
The latter equation follows from the obvious one
\[ \sigma^{-1}\bar\xi^C_2(\check\sff_1\sfv\tens\check\sff_1\sfv) =\sigma^{-1}\bar\xi^C_2 (\sigma\sfg_0|_{\bar{C}}\tens\sigma\sfg_0|_{\bar{C}}): \bar{C} \to \kk.
\]
Hence, the bijection
\begin{equation}
\UCCAlg(\bar{C}[-1]T^\ge,A)\rto\sim \CACoalg(C,\bar A[1]T^\ge)
\label{eq-UCCAlg-CACoalg}
\end{equation}
is constructed.

\begin{theorem}\label{thm-Cobar-Bar-adjoint-to-each-other}
The functors
\[ \Cobar: \CACoalg \leftrightarrows \UCCAlg: \Bbar 
\]
are adjoint to each other.
\end{theorem}

\begin{proof}
We have to prove naturality of bijection~\eqref{eq-UCCAlg-CACoalg} with respect to $A$ and $C$.
The bijection takes
\begin{equation}
(\sff_1,\und f) \mapsto (\check\sff_1\opr_A,\check\sff_1\sfv_A,\sigma\und f) \rMapsTo^{[1]} (\sigma^{-1}\check\sff_1\opr_A\sigma,\sigma^{-1}\check\sff_1\sfv_A\sigma,\und f\sigma) =(\check\sfg_1,\sfg_0|_{\bar{C}}\sigma,\sfw\sfg_0\sigma) \mapsto (\sfg_1,\sfg_0\sigma),
\label{eq-ff-fprfvsf-sfprssfvsfs-ggswgs-ggs}
\end{equation}
where
\begin{gather*}
\sfg_0 =(\sfg_0|_{\bar{C}},\sfw\sfg_0) =(\sigma^{-1}\check\sff_1\sfv_A,\und f): C =\bar C\oplus\kk \to \kk,
\\
\sfg_1 =\bar\Delta_C^{(k)} \cdot \check\sfg_1^{\tens k} =\bar\Delta_C^{(k)} \cdot (\sigma^{-1}\check\sff_1\opr_A\sigma)^{\tens k}: \bar C \to \bar A[1]^{\tens k}, \qquad k>0.
\end{gather*}

Naturality of \eqref{eq-UCCAlg-CACoalg} with respect to $A$ means that for each \(\sfh:A\to B\in\UCCAlg\)
\begin{diagram}[LaTeXeqno]
\UCCAlg(\bar{C}[-1]T^\ge,A) &\rTTo^\sim &\CACoalg(C,\bar A[1]T^\ge)
\\
\dTTo<{\UCCAlg(1,\sfh)} &= &\dTTo>{\CACoalg(C,\sfBar h)}
\\
\UCCAlg(\bar{C}[-1]T^\ge,B) &\rTTo^\sim &\CACoalg(C,\bar B[1]T^\ge)
\label{dia-UCCAlg(1h)-CACoalg(Ch)}
\end{diagram}
The left-bottom path takes \((\sff_1,\und f)\) to
\begin{multline*}
(\sff_1\sfh_1,\und f+\und h) \mapsto (\check\sff_1\sfh_1\opr_B,\check\sff_1\sfh_1\sfv_B,\sigma(\und f+\und h))
\\
\rMapsTo^{[1]} (\sigma^{-1}\check\sff_1\sfh_1\opr_B\sigma,\sigma^{-1}\check\sff_1\sfh_1\sfv_B\sigma,(\und f+\und h)\sigma) \mapsto (\sfq_1,\sfq_0\sigma),
\end{multline*}
where
\begin{equation*}
\sfq_1 =\bar\Delta_C^{(k)} \cdot (\sigma^{-1}\check\sff_1\sfh_1\opr_B\sigma)^{\tens k}: \bar C \to \bar B[1]^{\tens k}, \quad \sfq_0 =(\sigma^{-1}\check\sff_1\sfh_1\sfv_B,\und f+\und h): C =\bar C\oplus\kk \to \kk.
\end{equation*}
The top bijection takes \((\sff_1,\und f)\) to \eqref{eq-ff-fprfvsf-sfprssfvsfs-ggswgs-ggs} and the right morphism takes it to
\(\bigl(\sfg_1\cdot\Bbar_1h,(\sfg_0+(\eps_C\oplus\sfg_1)\Bbar_0h)\sigma\bigr)\).
We have
\begin{align*}
\sfg_1\cdot\Bbar_1h =\bar\Delta_C^{(k)} \cdot \check\sfg_1^{\tens k} \cdot \bar{h}_1^{\tens k} &= \bar\Delta_C^{(k)} (\sigma^{-1}\check\sff_1\opr_A\sigma\bar{h}_1)^{\tens k} =\bar\Delta_C^{(k)} (\sigma^{-1}\check\sff_1(1-\sfv\eta)\sfh_1\sigma\opr_B)^{\tens k}
\\
&=\bar\Delta_C^{(k)} (\sigma^{-1}\check\sff_1\sfh_1\opr_B\sigma)^{\tens k} =\sfq_1: \bar C \to B[1]^{\tens k},
\\
\sfg_0+\sfg_1\Bbar_0h &= \sigma^{-1}\check\sff_1\sfv_A +\sigma^{-1}\check\sff_1\opr_A\sigma h_1\bv_B =\sigma^{-1}\check\sff_1\sfh_1\sfv_B =\sfq_0: \bar C \to \kk
\end{align*}
due to obvious identity
\[ \sfv_A +\opr_A\sfh_1\sfv_B =\sfh_1\sfv_B: A \to \kk.
\]
Furthermore,
\[ \sfw\sfg_0+(\sfw,0)(\eps_C\oplus\sfg_1)\Bbar_0h =\und f+h_0\bv_B =\und f+\und h =\sfw\sfq_0: \kk \to \kk
\]
due to computation
\[ h_0\bv_B =\sfh_0\sfv_B =\und h\eta_B\sfv_B =\und h.
\]
Therefore, equation~\eqref{dia-UCCAlg(1h)-CACoalg(Ch)} is proven.

Naturality of \eqref{eq-UCCAlg-CACoalg} with respect to $C$ means that for each \(\sfj:C\to D\in\CACoalg\)
\begin{diagram}[LaTeXeqno]
\UCCAlg(\bar{D}[-1]T^\ge,A) &\rTTo^\sim &\CACoalg(D,\bar A[1]T^\ge)
\\
\dTTo<{\UCCAlg(\sfCobar j,A)} &= &\dTTo>{\CACoalg(\sfj,1)}
\\
\UCCAlg(\bar{C}[-1]T^\ge,A) &\rTTo^\sim &\CACoalg(C,\bar A[1]T^\ge)
\label{dia-UCCAlg(jA)-CACoalg(j1)}
\end{diagram}
The left-bottom path takes \((\sff_1,\und f)\) to
\begin{multline*}
((\sfCobar_1j)\cdot\sff_1,\und{\Cobar j}+\und f) \mapsto (j_1\check\sff_1\opr_A,j_0+j_1\check\sff_1\sfv_A,\sigma(\und{\Cobar j}+\und f))
\\
\rMapsTo^{[1]} (\sigma^{-1}j_1\check\sff_1\opr_A\sigma,\sigma^{-1}(j_0+j_1\check\sff_1\sfv_A)\sigma,(\und{\Cobar j}+\und f)\sigma) \mapsto (\sfr_1,\sfr_0\sigma),
\end{multline*}
which takes into an account that
\[ ((\sfCobar_1j)\cdot\sff_1)\spcheck =j_0\eta_A +j_1\check\sff_1: \bar{C}[-1] \to A.
\]
The top bijection takes \((\sff_1,\und f)\) to \((\sfg_1,\sfg_0\sigma)\) from \eqref{eq-ff-fprfvsf-sfprssfvsfs-ggswgs-ggs} and the right morphism takes it to \((\sfj_1\sfg_1,(\sfj_0+\sfj_1\sfg_0)\sigma)\).
This coincides with \((\sfr_1,\sfr_0\sigma)\).
In fact,
\begin{align*}
\sfj_1\sfg_1 =\sfj_1 \bar\Delta_D^{(k)} (\sigma^{-1}\check\sff_1\opr_A\sigma)^{\tens k} =\bar\Delta_C^{(k)} (\sigma^{-1}j_1\check\sff_1\opr_A\sigma)^{\tens k} =\sfr_1 &: \bar C \to \bar A[1]^{\tens k},
\\
\sfj_0+\sfj_1\sfg_0 =\sigma^{-1}j_0 +\sfj_1\sigma^{-1}\check\sff_1\sfv_A =\sigma^{-1}(j_0+j_1\check\sff_1\sfv_A) =\sfr_0 &: \bar C \to \kk,
\\
\sfw_C\sfj_0+\sfw_C\sfj_1\sfg_0 =\sfw_C\sfj_0+\sfw_D\sfg_0 =\und{\Cobar j}+\und f =\sfw\sfr_0 &: \kk \to \kk.
\end{align*}
Therefore equation~\eqref{dia-UCCAlg(jA)-CACoalg(j1)} holds and the theorem is proven.
\end{proof}

Notice that the both sides of \eqref{eq-f-wg} do not appear in the equations at all.
One may assume that the components of morphisms of curved algebras and coalgebras belonging to $\kk^1$ are all 0.
Then one gets subcategories $\uccAlg\subset\UCCAlg$ and $\caCoalg\subset\CACoalg$ with smaller sets of morphisms.

\begin{definition}
Objects of the category $\uccAlg$ are unit-complemented curved algebras and morphisms are graded algebra homomorphisms \(\sff:A\to B\) such that
\begin{equation*}
\sff m^B_1 =m^A_1\sff, \qquad m^B_0 =m^A_0\sff.
\end{equation*}
The composition and the identity morphisms are inherited from $\gr\alg$.
\end{definition}

\begin{definition}
Objects of the category $\caCoalg$ are curved augmented coalgebras and morphisms \(\sfg:C\to D\) are pairs \((\sfg_1,\sfg_0')\) consisting of a homomorphism of augmented graded coalgebras \(\sfg_1:C\to D\) and a $\kk$\n-linear map \(\sfg_0':\bar C\to\kk\) of degree~1 such that
\begin{align*}
\delta^C_1\sfg_1 +\opr_C\bar\delta^C_2(\sfg_0'\tens\sfg_1 -\sfg_1\tens\sfg_0') &=\sfg_1\delta^D_1: C \to \bar D,
\\
\delta^C_0 -\delta^C_1\sfg_0' -\opr_C\bar\delta^C_2(\sfg_0'\tens\sfg_0') &=\sfg_1\delta^D_0: C \to \kk.
\end{align*}
The composition $\sfh:C\to E$ of morphisms \(\sff:C\to D\) and \(\sfg:D\to E\) is given by \(\sfh_1=\sff_1\sfg_1\), \(\sfh_0'=\sff_0'+\sff_1\sfg_0'\).
The unit morphism is \((\id,0)\).
\end{definition}

Notice that there is a functor \(\Bbar:\uccAlg\to\caCoalg\), making the diagram of functors
\begin{diagram}
\uccAlg &\rTTo^\Bbar &\caCoalg
\\
\dTTo &= &\dTTo
\\
\UCCAlg &\rTTo^\Bbar &\CACoalg
\end{diagram}
commute on the nose.
In view of \eqref{eq-f1f0-A1k} $\Bbar$ takes a morphism \(f:A\to B\in\uccAlg\) to the strict coalgebra morphism
\[ \sfBar_1f =\bar f_1T^\ge: \bar{A}[1]T^\ge \to \bar{B}[1]T^\ge,
\]
and the degree~1 functional is
\[ \sfBar_0'f =\bigl( \bar{A}[1]T^> \rTTo^{\pr_1} \bar{A}[1] \rMono A[1] \rTTo^f B[1] \rTTo^\bv \kk \bigr).
\]
The restriction of \eqref{eq-f0-g0-A1T} to $\kk$ vanishes.

Also there is a functor \(\Cobar:\caCoalg\to\uccAlg\), making commutative the diagram of functors
\begin{diagram}
\caCoalg &\rTTo^\Cobar &\uccAlg
\\
\dTTo &= &\dTTo
\\
\CACoalg &\rTTo^\Cobar &\UCCAlg
\end{diagram}
It takes a morphism \(\sfg=(\sfg_1,\sfg_0'):C\to D\in\caCoalg\) to the algebra homomorphism
\[ \sfCobar g: \bar{C}[-1]T^\ge \to \bar{D}[-1]T^\ge,
\]
specified by its components
\[ \bar g_1 =g_1|: \bar{C}[-1] \to \bar{D}[-1], \qquad g_0': \bar{C}[-1] \to \kk,
\]
which coincides with \eqref{eq-g1g0g0}.
If $\und g=0$, then $\und{\Cobar g}$ given by \eqref{eq-und-Cobar-g-und-f} vanishes as well.
Since equations \eqref{eq-fm-mf-m-mf}, \eqref{eq-dg-dg11gg-gd-d-dg} distinguishing morphisms in diagram~\eqref{dia-f-g-fsf-ggs} do not involve $\und f$, $\und g$, we have

\begin{corollary}[to \thmref{thm-Cobar-Bar-adjoint-to-each-other}]
The functors
\[ \Cobar: \caCoalg \leftrightarrows \uccAlg: \Bbar 
\]
are adjoint to each other.
\end{corollary}

Now let us describe full subcategories of the above categories.

\begin{definition}
A \emph{unit-complemented $\dg$-algebra} is a unit-complemented curved algebra \((A,m_2,m_1,0,\eta,\sfv)\) with $m_0=0$.
Equivalently, it is a $\dg$\n-algebra \((A,m_2,m_1,\eta)\) with a degree 0 map \(\sfv:A\to\kk\) (splitting of the unit) such that \(\eta\cdot\sfv=1_\kk\).
Morphisms of such algebras are morphisms of $\dg$\n-algebras.
Their full subcategory is denoted \(\ucdgAlg\subset\uccAlg\).
\end{definition}

\begin{definition}
\emph{Augmented curved coalgebras} are defined as curved augmented coalgebras \((C,\delta_2,\delta_1,\delta_0,\eps,\sfw)\) with
\begin{equation}
\sfw\delta_1 =0, \qquad \sfw\delta_0 =0.
\label{eq-wd10-wd00}
\end{equation}
The full subcategory of such coalgebras is denoted $\acCoalg\subset\caCoalg$.
\end{definition}

Positselski \cite{0905.2621} formulates equations~\eqref{eq-wd10-wd00} as \((\sfw,0):\kk\to C\) being a morphism in $\caCoalg$.
Clearly, \(\Cobar(\Ob\acCoalg)\subset\Ob\ucdgAlg\).

\begin{proposition}
The functor $\Bbar$ restricts to a functor \(\Bbar:\ucdgAlg\to\acCoalg\), which has a left adjoint.
The adjunction is \(\Cobar:\acCoalg\leftrightarrows\ucdgAlg:\Bbar\).
\end{proposition}

\begin{proof}
We have to prove that \(\Bbar(\Ob\ucdgAlg)\subset\Ob\acCoalg\).
This follows from two remarks.
First,
\[ \sfw^{\Bbar A}\delta^{\Bbar A}_1 =\bigl( \kk \rMono^{\inj_0} \bar{A}[1]T^\ge \rTTo^{\bar{b}} \bar{A}[1]T^\ge \bigr) =0
\]
since
\[ \bar{b} =\sum_{a+k+c=n}^{k>0} \bigl(1^{\tens a}\tens\bar b_k\tens1^{\tens c}: \bar A[1]^{\tens n} \to \bar A[1]^{\tens(a+1+c)} \bigr).
\]
Second,
\[ \sfw^{\Bbar A}\delta^{\Bbar A}_0 =-\bigl( \kk \rMono^{\inj_0} A[1]T^\ge \rTTo^{\check{b}} A[1] \rTTo^\bv \kk \bigr) =0
\]
since $b_0=0$.
\end{proof}

\subsection{Twisting cochains.}
Let us consider a unit-complemented curved algebra $A$ and a curved augmented coalgebra $C$.
A morphism \(f\in\uccAlg(\bar{C}[-1]T^\ge,A)\) is identified with a degree~0 map \(\check\sff_1:\bar{C}[-1]\to A\) which satisfies equations~\eqref{eq-fm-xe-xf-xffm}.
Equivalently, the degree~1 map \(\theta:C\to A\) satisfies the equations
\begin{subequations}
\begin{align}
\sfw\theta &= 0: \kk \to A,
\label{eq-theta-a}
\\
\theta m_1^A +\delta_1\theta &= \delta_0\eta^A +\eps^Cm^A_0 -\delta_2(\theta\tens\theta)m^A_2: C \to A.
\label{eq-theta-b}
\end{align}
\label{eq-theta}
\end{subequations}
In fact, each solution of \eqref{eq-theta-a} has the form
\[ \theta =\bigl\langle C \rTTo^{\opr_C} \bar C \rTTo^{\sigma^{-1}} \bar C[-1] \rTTo^{\check\sff_1} A \bigr\rangle
\]
for a unique $\check\sff_1$.
Restricting \eqref{eq-theta-b} to $\bar C$ gives the top equation from \eqref{eq-fm-xe-xf-xffm}, while restricting to the image of $\sfw:\kk\to C$ gives the bottom equation from \eqref{eq-fm-xe-xf-xffm}.

\begin{definition}
The degree~1 map \(\theta:C\to A\) that satisfies \eqref{eq-theta} is called a \emph{twisting cochain}.
\end{definition}

The set \(\Tw(C,A)=\{\textup{twisting cochains }\theta:C\to A\}\) is in bijection with the homomorphism sets
\begin{diagram}
\uccAlg(\bar{C}[-1]T^\ge,A) &\rTTo^\sim &\Tw(C,A) &\rTTo^\sim &\caCoalg(C,\bar A[1]T^\ge),
\\
\check\sff_1 &\rMapsTo &\opr_C\sigma^{-1}\check\sff_1 =\theta &\rMapsTo &(\theta\big|_{\bar{C}}\sigma,\theta\big|_{\bar{C}}\sfv) =(\check\sfg_1,\check\sfg_0').
\end{diagram}

When $A$ is a unit-complemented $\dg$\n-algebra and $C$ is an augmented curved coalgebra the notion of a twisting cochain simplifies to a degree~1 map \(\theta:\bar C\to A\) which satisfies the equation
\[ \theta m_1^A +\delta_1\theta = \delta_0\eta^A -\delta_2(\opr_C\tens\opr_C)(\theta\tens\theta)m^A_2: \bar C \to A.
\]

\subsection{Conclusion.}
The results of the article indicate that a dual notion to differential graded algebra is the augmented curved coalgebra, and not a differential graded coalgebra as one might think a priori.

\ifx\chooseClass1
{\footnotesize

\vsk

Institute of Mathematics NASU,
3 Tereshchenkivska st.,
Kyiv-4, 01601 MSP,
Ukraine\\
lub@imath.kiev.ua

\vs
}
\received{?} %(to fill in by the Editors)

%%\revised{?} % (to fill in by the Editors)
	\else

%\bibliography{yuri}
\tableofcontents
\fi
\end{document}